\documentclass[a4]{article}

\usepackage{amsmath,amsfonts,amsthm,amssymb,graphicx,tikz-cd,stmaryrd}
\theoremstyle{definition}
\newtheorem{thm}{Theorem}[section]
\newtheorem{deff}[thm]{Definition}

\newtheorem{prop}[thm]{Proposition}
\newtheorem{cor}[thm]{Corollary}
\newtheorem{rem}[thm]{Remark}

\newcommand{\zahl}{\mathbb{Z}}

\newcommand{\cpx}{\mathbb{C}}

\newcommand{\ratmap}{\dashrightarrow}

\newcommand{\sym}{\mathfrak{S}}

\newcommand{\pone}{\mathbb{P}^1}
\newcommand{\proj}{\mathbb{P}}

\newcommand{\fct}[1]{\operatorname{#1}}

\newcommand{\cat}[1]{\mathcal{#1}}
\newcommand{\shf}[1]{\mathcal{#1}}
\newcommand{\oshf}{\shf{O}}
\newcommand{\ishf}{\shf{I}}
\newcommand{\lshf}{\shf{L}}

\newcommand{\id}{\fct{id}}

\newcommand{\dyn}{\fct{Dyn}}
\newcommand{\ccor}{\fct{Corr}}

\newcommand{\glt}{\fct{GL}_2}
\newcommand{\slt}{\fct{SL}_2}
\newcommand{\pglt}{\fct{PGL}_2}
\newcommand{\af}{\mathbb{A}}
\newcommand{\aff}{\mathbb{A}}
\newcommand{\gmult}{\mathbb{G}_m}

\title{Dynamical Systems of Correspondences on the Projective Line I: Moduli Spaces and Multiplier Maps}
\author{Rin Gotou\thanks{Department of Mathematics, Graduate School of Osaka University \texttt{u661233h@ecs.osaka-u.ac.jp}}}
\date{}

\begin{document}
\maketitle
\begin{abstract} We consider moduli spaces of dynamical systems of correspondences over the projective line as a generalization of moduli spaces of dynamical systems of endomorphisms on the projective line. We obtain the rationality of the moduli spaces. The rationality of the moduli space of degree $(d,e)$ correspondences is obtained by constructing a projection to the one for the usual dynamical systems of degree $d+e-1$ representation-theoretically.
We also show that the multiplier maps for the fixed points and the multiplier index theorem (Woods Hole formula) are also reduced through the projection and obtain the reduced form explicitly.
\end{abstract}
\section{Introduction}

Moduli spaces of dynamical systems over the projective line $\pone$, which parameterizes endomorphisms up to the conjugations by the automorphisms on $\pone$, are considered in \cite{Sil} by using geometric invariant theory (GIT, \cite{MF}). 

Self-correspondence is a generalization of endomorphism. Some important concepts on a dynamical system of endomorphism have natural generalization for a dynamical system of (self-)correspondence. An example is Woods Hole formula, which is originally stated for correspondence in \cite{A-B} and \cite{Ill} and have used for dynamical system of self-maps from \cite{Ue}. Other examples are the canonical measure and the canonical height, they are originally stated for self-map and generalized to correspondence in \cite{D-K-W} and \cite{In} respectively.

In this paper, we construct moduli spaces for dynamical systems of correspondences on the projective line as an analogue of \cite{Sil}.


We firstly construct the moduli space $\ccor_{d,e}$ of correspondences of degree $(d,e)$, which parameterizes the closed subschemes $C \subset \pone_x \times \pone_y $ defined by an equation $\sum_{i=0}^d \sum_{j = 0}^e a_{ij}x^iy^j = 0 $. We consider the diagonal action of $\fct{Aut}(\pone) \simeq \pglt$ on $\pone \times \pone$, which is equivalent to the conjugation action on the graph variety (\cite{Sch}) and give characterization of stable points and semistable points respectively. The loci were given in \cite{Sil} for the case $d=1$, that is, the case of moduli spaces of rational maps. We obtain a simple generalization for our moduli spaces of correspondences.
\begin{thm} The point of $\ccor_{d,e}$ which represents a correspondence $C$ is a stable point (resp. a semistable point) if and only if $C$ has no point of multiplicity $\geq \frac{d+e}{2}$ (resp. of multiplicity $> \frac{d+e}{2}$) on the diagonal of $\pone \times \pone$.
\end{thm}
\begin{cor} $\ccor_{d,e}^{ss} = \ccor_{d,e}^{s}$ if and only if $d+e$ is odd.
\end{cor}
\begin{rem} Geometric invariant theory (GIT, \cite{MF}) ensures the existence of a uniform geometrical quotient of the stable locus and a compactification of the quotient as a universal categorical quotient of the semistable locus.
\end{rem}

The compactified moduli of dynamical systems $\dyn_{d,e} := \ccor_{d,e}^{ss} \sslash \pglt $ is constructed as the projective spectrum of a graded invariant ring. Computational invariant theory (CIT, \cite{D-K}) gives precise information for an invariant ring by treating degree-wise actions representation-theoretically before taking the invariant subring. An application of CIT for the moduli space of rational maps is done in \cite{We}.

In this paper, we construct (a $\af^1 \setminus \{ 0 \}$-family of) rational maps \[ \rho_c : \dyn_{d,e} \ratmap \dyn_{1,d+e-1} \]  from representation-theoretic (classically well-known, Clebsch-Gordan) decomposition. The rationality of the moduli $\dyn_{1,d+e-1}$ is shown in \cite{Le} and argument in \cite{Le} also leads that these rational maps $\rho_c$ are generically affine space bundle, therefore we can generalize the rationality result for $\dyn_{d,e}$.
\begin{prop}(Proposition \ref{rational}) $\dyn_{d,e}$ is rational for $d, e \geq 1$ and $(d,e) \neq (1,1)$.
\end{prop}

An application of the moduli space is about an inverse problem, a problem which asks the existence and the classification of dynamical systems with proscribed invariants. An example of such parameter is the set of the multipliers of periodic orbits.

For a dynamical system $f : \pone \to \pone$, we denote the elementary symmetric polynomials of the fixed point multipliers by $\sigma_{k}(f)$, that is, \begin{equation} 1 + \sum_{i = 1}^{d+1} \sigma_{k}(f)t^k =\prod_{x : f(x) = x} \left( 1 + f'(x)t \right)  \label{mlt} \end{equation} for a formal variable $t$. The rational map \[ \lambda_{1,(1,d)}: \dyn_{1,d} \ratmap \proj^{d+1},\ \lambda_{1,(1,d)}([f]) := [1:\sigma_1(f) : \cdots : \sigma_{d+1}(f)] \] is called {\em the fixed point multiplier map}, which is used to solve the rationality of the moduli $\dyn_{1,2}$, (\cite{Mi}, \cite{Sil}), and also to state and to consider inverse problems for multipliers (\cite{Fu1},\cite{Fu2},\cite{H-T},\cite{Mc},\cite{Su}).

A fundamental relation among multipliers was obtained as an application of the Woods Hole Formula (\cite{Sil2},\cite{Ue}) \begin{equation} \sum_{x : f(x) = x} \frac{1}{1-f'(x)} = 1, \end{equation}
or equivalently, 
\begin{equation}
\sum_{i= 0}^{d+1}(-1)^i (d-i)\sigma_i = 0. \label{index}
\end{equation}

In this paper, we construct multiplier maps for correspondence \[\lambda_{n,(d,e)}: \dyn_{d,e} \ratmap \proj^{d+e},\] where we consider the fixed points as the points $\{ z \in \pone | f(z,z) = 0 \}$ for a correspondence defined by $f(x,y) = 0$. The Woods Hole formula for correspondence is obtained in \cite{A-B} and \cite{Ill}, but the reduction using Clebsch-Gordan decomposition gives an alternative proof. We can see that the fixed point multiplier map for correspondences are also reduced through the rational maps $\dyn_{d,e} \ratmap \dyn_{1,d+e-1}$ used to show the rationality of $\fct{Dyn}_{d,e}$, that is,
\begin{prop}(Proposition \ref{unif1}) There exists a projective linear morphism $A \in \fct{Aut} (\proj^{d+e}) = \fct{PGL}_{d+e}$ which makes the following diagram commutative:
\[
\begin{tikzcd} \fct{Dyn}_{d,e} \arrow[rr,dashed]{}{ \lambda_{1,(d,e)} } \arrow[d,dashed] & & \proj^{d+e} \arrow{d}{A} \\
\fct{Dyn}_{1,d+e-1} \arrow[rr,dashed]{}{\lambda_{1,(1,d+e-1)}} & & \proj^{d+e}.
\end{tikzcd}
\]
\end{prop}

Representation theoretically, we can obtain the following formula about resultant from the Woods Hole formula:
\begin{thm}(Corollary \ref{repWH}) For an arbitrary field $k$ and any polynomials $f,g \in k[x]$ such that $\deg f \geq 3$ and $\deg f \geq \deg g +2$, 
\[ \left. \frac{\partial }{\partial t} \fct{res}_x(f(x) , f'(x) + t g(x)) \right|_{t = 0} = 0. \]
\end{thm}

This paper is organized as follows: In Section 2, we set up notation and terminology. In Section 3, we review representation theory of linear groups including Clebsch-Gordan decomposition that we use later. In Section 4, we construct moduli spaces of correspondences and rewrite the composition and conjugation of correspondences as maps and actions on the moduli spaces. In Section 5, we construct the moduli spaces of dynamical systems of correspondences, characterize the stable and semistable loci of the conjugation action and show the rationality of the moduli spaces. In Section 6, we construct multiplier maps and reformulate the multiplier index theorem representation-theoretically.

\subsection*{Acknowledgments}
I would like to thank Seidai Yasuda for proofreading this paper, originally the former half of my master thesis and for having supervised me in the master course. I am grateful to Takehiko Yasuda for proofreading this paper, having useful discussion and having accepted me as a master course and doctor course student.

\subsection*{Anacknowledgment}
The author and this work are supported by JSPS KAKENHI Grant-in-Aid for Research Fellow JP202122197.

\section{Notation and Terminology}

Throughout this paper, we follow \cite{Liu} for the terminology of algebraic geometry.

We fix a field $k$ of characteristic zero. Unless otherwise stated, we suppose that every scheme is a scheme over $k$.  

For a ring $R$ and a free $R$-module $M$ of finite rank, we denote by $R[M]$ the symmetric algebra $\fct{Sym}^\bullet_R M  := \bigoplus_{n=0}^{\infty} M^{\otimes n}/\langle v \otimes w - w \otimes v \rangle$. We note that if we take an $R$-basis $\{ x_1,\ldots , x_r \}$ of $M$, $R[M]$ is canonically isomorphic to the polynomial ring $R[x_1,\ldots , x_r]$.
If a group $G$ and a representation $\rho : G \to \fct{Aut}_R (M)$ are also given, we write $R[M]^G$ for the invariant ring.

We denote by $\fct{Sym}_n M$ the permutation-invariant part of the $n$-th tensor power. 

\section{Representation Theory for $\slt$ and $\pglt$}

In this section, we provide representation-theoretic terminology. We only use classically well-known operations, which are Cayley operators and Clebsch-Gordan decomposition. Contents in this section are found in, for example, \cite{Ca}.
 
For the special linear group $\slt(k)$, we denote the trivial representation on $k$ by $V_0$, the natural representation on $k^2$ by $V_1$ and the symmetric tensor representation $\fct{Sym}_n(V_1)$ by $V_n$. From the construction, we have $\dim V_n = n+1$.

\begin{prop}\label{even} If $n$ is an even number, then there exists an action of $\pglt (k)$ on the vector space $V_n$ such that the pullback action to $\slt (k) $ is $V_n$.
\end{prop}
\begin{proof}We consider the representation $\Delta_n(-\frac{n}{2})$ of $\fct{GL}_2(k)$ on $(k^2)^{\otimes n}$ which is given by
\[ g \cdot (v_1 \otimes v_2 \otimes \cdots \otimes v_n) := (\det g )^{-\frac{n}{2}}(gv_1 \otimes gv_2 \otimes \cdots \otimes gv_n)\ (\text{for }g \in \glt (k) ).\]
By $\Delta_n(-\frac{n}{2})$, any scalar matrix $\binom{c\ \ 0}{0\ \ c} \in \fct{GL}_2(k)$ acts trivially, therefore $\Delta_n(-\frac{n}{2})$ is a representation of $\pglt (k)$. Moreover, the action of $\pglt(k)$ by $\Delta_n(-\frac{n}{2})$ commutes with the action of $\sym_n$ by permuting components, therefore we can obtain a subrepresentation of $\pglt (k)$ on $\fct{Sym}_n k^2$. The restriction of the subrepresentation on $\slt (k)$ is the $n$-th symmetric tensor of $V_1$, therefore this is a required representation.
\end{proof}
\begin{prop}\label{dualrep} The dual representation of $V_n$ is isomorphic to $V_n$.
\end{prop}
\begin{proof} Since $k$ is of characteristic $0$, it is enough to show the proposition for $n=1$ (see Remark \ref{caniso}).

Let $[e_1 , e_2]$ be the basis of $V_1$ and $[f_1 , f_2]$ the dual basis of $V_1^*$. For $A \in \slt (k)$, 
By the definition of dual representation, $A \cdot f_1 = f_1 \circ A^{-1}$. Therefore the transpose $(A^{-1})^T$ is the representation matrix of the action of $A$ by the dual representation. We have $(A^{-1})^T = I A I^{-1}$ for $I = \binom{0\ -1}{ 1\ \ 0}$, therefore the linear morphism such that $e_1 \mapsto -f_2,\ e_2 \mapsto f_1$ is an isomorphism.
\end{proof}
\begin{rem}\label{caniso} The canonical morphism $\fct{Sym}_{n} (V^*) \to (\fct{Sym}_{n} V )^*$ induced from the inclusion morphism \[ \fct{Sym}_n V = (V^{\otimes n})^{\sym_n } \to V^{\otimes n}\] is an isomorphism because our field $k$ has characteristic 0 and the binomial coefficients are invertible. In positive characteristic $p$, the canonical morphism $\fct{Sym}_{n} (V^*) \to (\fct{Sym}_{n} V )^*$ is not isomorphism for $n \geq p$. In fact, the representations $V_n = \fct{Sym}_n (V_1) $ and $V_n^* = \fct{Sym}_n (V_1)^*$ are not isomorphic (\cite{Ca}).
\end{rem}

\subsection{Weight}
For a group $G$, we denote the abelian category of representations of $G$ over $k$ by $\fct{Rep}_k (G)$. For a representation $V$ of a group $G$, we denote by $\fct{Rep}_k (V)$ the minimal full subcategory of $\fct{Rep}_k(G)$ includes $V$ as an object and whose objects are closed by taking finite tensor products, finite direct sums, dual representation and subrepresentations.
The $K_0$-group of a small abelian category $\cat{C}$ is defined by \[ K_0(\cat{C}) := \bigoplus_{V \in \fct{objclass}(\cat{C})} \zahl [V] {\bigg/} \binom{ [V] - [V'] - [V''] }{\text{for exact } 0 \to V' \to V \to V'' \to 0} \]
where $\fct{objclass}(\cat{C})$ is the set of isomorphism classes of objects of $\cat{C}$.
If the abelian category $\cat{C}$ is monoidal and the tensor functor is exact in both arguments, $K_0(\cat{C})$ becomes a ring by $[V] \cdot [W] := [V \otimes W]$. The induced product is commutative if the monoidal product is symmetric, therefore $K_0(\fct{Rep}_k(V))$ for any group $G$ and any representation $V$ is a commutative ring.
\begin{prop}\label{lau1} Let $k$ be an infinite field. For the representation $k(1) := (\fct{id} : k^\times \to \fct{GL}_1(k) , k)$ of $k^\times$, there exists an isomorphism $w : K_0( \fct{Rep}_k(k(1)) ) \simeq \zahl [ q, q^{-1} ]$ induced by $[k(1) ] \mapsto q$.
\end{prop}
\begin{proof} We write $k(n)$ for $k(1)^{\otimes n}$ and $k(-n)$ for $k(n)^*$ for nonnegative integer $n$. We define a ring morphism $w^{-1} : \zahl [q,q^{-1}] \to K_0(\fct{Rep}_k(k(1)))$ by $q \mapsto [k(1)]$. The morphism $w^{-1}$ is well-defined by sending $q^{-1}$ to $[k(-1)]$, because $k(0) := k(1) \otimes k(-1)$ is the trivial representation, which is the multiplicative unit of the ring $K_0(\fct{Rep}_k(k(1)))$.

We will show $w^{-1}$ is bijective.

(injectivity) For distinct integers $m\neq n$, there exists $a \in k^\times $ such that $a^n \neq a^m$ since $k$ is infinite. Therefore $k(n)$ and $k(m)$ are not isomorphic to each other. This leads to that $w^{-1}$ is injective.

(surjectivity) We show that every subrepresentation $W$ of a representation $V := \bigoplus_{i=1}^{m} k(n_i)^{\oplus d_i }$ for distinct integers $n_i$ and positive integer $m$ is isomorphic to $\bigoplus_{j=1}^{m'} k(n'_j)^{\oplus d'_j}$ for some $n'_i$s, $d'_i$'s and $m'$. Let $p_i : V \to k(n_i)^{\oplus d_i}$ be the projections. It is enough to show that for any vector $v \in W$, $p_i(v) \in W$. 

Let $v \in W$ be a vector. We denote the action of $a \in k^\times$ by $a \cdot v$.
Since $k$ is infinite, we can take $a \in k^\times$ as the map $\{ n_1 , \ldots , n_m \} \to k^\times : n_i \mapsto a^{n_i}$ is injective (where $\{ n_1 , \ldots , n_m \}$ is considered as a set). Then a linear map
\[ k^{\oplus m} \to \sum_{i = 1}^{m} k p_i(v) : (b_1, b_2, \cdots ,b_n) \mapsto \sum_{i=1}^{m} b_i(a^{i-1} \cdot v) \]
is surjective because the representation matrix via the basis $(b_1, b_2 ,\ldots , b_n)$ and $(p_{i_1}(v) , p_{i_2}(v) , \ldots , p_{i_{m'}}(v))$ is full rank because it is a submatrix of the Vandermonde matrix of $\{ a^{n_1} , \ldots , a^{n_m} \}$, here $(p_{i_1}(v) , p_{i_2}(v) , \ldots , p_{i_{m'}}(v))$ is made by picking nonzero components out of $(p_1(v) , \ldots , p_n(v))$. Therefore we have $p_i(v) = \sum_{i=1}^{m} b_i(a^{i-1} \cdot v)$ for some $b_i \in k$, that is $p_i(v) \in W$ because $W$ is a subrepresentation.
\end{proof}
\begin{deff} For a group $G$, a group morphism $\lambda : k^\times \to G$ and a finite dimensional $G$-representation $V$ such that the $k^\times$-representation of $V$ induced from $V$ and $\lambda$ is in $\fct{Rep}_k(k(1))$, we call the composition morphism \[K_0 (\fct{Rep}_k(V)) \to K_0(\fct{Rep}_k(k(1))) \simeq \zahl[q,q^{-1}]\] the {\em weight morphism associated to $\lambda$} and denote by $w_{\lambda}$.
\end{deff}
In the notation of the above definition, the Laurent polynomial $w_\lambda ([W] )$ for any representation $W \in \fct{Rep}_k (V)$ has positive coefficients.

We fix the group morphism
\[ c : k^\times \ni t \mapsto \binom{t\ \ \ 0\ }{0\ \ t^{-1}} \in \slt (k). \]
It immediately follows that the $k^\times$-representation induced from $V_1$ and $c$ is in $\fct{Rep}_k(k(1))$ and $w_c(V_1) = q + q^{-1}$.

\subsection{Clebsch-Gordan Decomposition}
For a finite dimensional vector space $V$, the space of $n$-ary forms, that is, the vector space of the all degree $n$ homogeneous polynomials in $k[V]$ is naturally isomorphic to $(\fct{Sym}_n(V^*))^*$ in arbitrary characteristic. Therefore the representation $V_n$ is identified with the space of $n$-ary binomial forms in characteristic zero. The variables are the standard basis of $V_1 = k^2$ indeed.
\begin{prop}\label{irr} Let $k$ be a field of characteristic zero. Then the representation $V_n$ is irreducible.
\end{prop}
\begin{proof} We note that $c(t) \cdot x_0^{n-i}x_1^i = t^{n-2i}x_0^{n-i}x_1^i$ under the identification between $V_n$ and the space of $n$-ary binomial forms, where the pair $(x_0,x_1)$ is the standard basis of $k^2$.

Let $W$ be an arbitrary nonzero $\slt$-stable subspace of $V_n$. By Proposition \ref{lau1}, we have
\[ W = \bigoplus_{i \in I_W} k x_0^{n-i}x_1^i \]
for some nonempty $I_W \subset \{ 0 , 1, \ldots , n \}$. We take a monomial $x_0^{n-i} x_1^i \in W$. Since $W$ is $\slt$-stable, we have $\binom{1\ 1}{1\ 2} \cdot x_0^{n-i} x_1^i = (x_0+x_1)^{n-i}(x_0 + 2x_1)^i \in W$. Therefore we have $I_W \supset \{ 0, \ldots, n \} $, that is, $V=W$.
\end{proof}

Under the identification, if we have two binomial forms $f(x_0 , x_1)$ and $g(y_0,y_1)$ of degree $d$ and $e$ respectively, the Cayley operator $\Omega_{xy} :=  \partial_{x_0}\partial_{y_1} - \partial_{y_0} \partial_{x_1}$ gives a new binomial form 
\[ \left. \left( \Omega_{xy}^m f(x_0 , x_1) \cdot g(y_0,y_1) \right) \right|_{(x_0 , x_1) = (y_0 ,y_1) = (z_0,z_1)} \]
of variables $(z_0,z_1)$ and degree $d+e-2m$, for $0 \leq m \leq \min (d,e) = \frac{1}{2}(d+e-|d-e|) $. This linear map is $\slt$-equivariant, that is, we have a morphism of representation
\[ \Omega^m : V_d \otimes V_e \to V_{d+e-2m} \]
given by
\begin{equation} \Omega^m( f(x_0,x_1) \otimes g(y_0,y_1) ) :=  \left. \left( \Omega_{xy}^m f(x_0 , x_1) \cdot g(y_0,y_1) \right) \right|_{(x_0 , x_1) = (y_0 ,y_1) = (z_0,z_1)}. \label{omega} \end{equation}

\begin{prop}\label{CGdec}(Clebsch-Gordan Decomposition, \cite[Theorem 3.2.4]{Ca}) If $\fct{char} k = 0$, the morphism 
\[ \bigoplus_{i=0}^{|d-e|} \Omega^i : V_d \otimes V_e \to V_{d+e} \oplus V_{d+e-2} \oplus \cdots \oplus V_{|d-e|+2} \oplus V_{|d-e|} \]
is an isomorphism.
\end{prop}
\begin{proof}
It is enough to show that the morphism $ \Omega^m : V_d \otimes V_e \to V_{d+e-2m} $ is surjective for $0 \leq m \leq \min (d,e)$. By Proposition \ref{irr}, it is sufficient to show that the morphism $ \Omega^m : V_d \otimes V_e \to V_{d+e-2m} $ is nonzero, which follows from the explicit calculation
\[ \Omega^m (x_0^d  \otimes y_0^{e-m}y_1^m ) = \frac{d!m!}{(d-m)! } z_0^{d+e-2m}. \]
\end{proof}

\section{Correspondence}
Let $X$ and $Y$ be schemes. A closed subscheme of $X \times Y$ is called an algebraic correspondence between $X$ and $Y$. We can regard a morphism $X \to Y$ as an algebraic correspondence given by the graph of the morphism. Therefore algebraic correspondence is a generalization of morphism.

\subsection{The Moduli Space of Correspondences over $\proj^1$}
Over $\pone \times \pone$, we denote the line bundle $p_1^* \oshf(d) \otimes p_2^* \oshf(e)$ by $\oshf(d,e)$ and the set of its global sections by $V_{d,e}$. We fix homogeneous coordinates of each components, choosing bases $x_0,x_1 \in V_{1,0}$ and $y_0,y_1 \in V_{0,1}$. Using these coordinates we have
\[ V_{d,e} = \left\{ \left. \sum_{0 \leq i \leq d,\ 0 \leq j \leq e} a_{i,j}x_0^{d-i}x_1^{i}y_0^{e-j}y_1^j \right| a_{i,j} \in k \right\}. \]
Later we will see that we can identify $V_{d,e}$ and $V_d \otimes V_e$ (Corollary \ref{conjact}).
\begin{deff} A nonzero element $f$ of $V_{d,e}$ is called {\em a divisorial quasi-correspondence} of degree $(d,e)$. A $(d+1) \times (e+1)$-matrix $A = (a_{i,j})_{0\leq i \leq d, \\ 0 \leq j \leq e}$ is called {\em the coefficient matrix} of $f$ if
\begin{equation} f = \sum_{0 \leq i \leq d,\ 0 \leq j \leq e} a_{i,j}x_0^{d-i}x_1^{i}y_0^{e-j}y_1^j. \label{divqc} \end{equation}
\end{deff}
We often abbreviate a divisorial quasi-correspondence (\ref{divqc}) as $f = \sum_{i , j} a_{ij} x^i y^j$ if the degree of correspondence $(d,e)$ is apparent.
\begin{deff} A closed subscheme $C$ of $\pone \times \pone$ is said to be {\em a divisorial correspondence} if $\oshf_C = \oshf_{\pone \times \pone }/\shf{I}$ and $\shf{I}$ is a locally free sheaf of rank one. A divisorial correspondence $C$ given by an ideal sheaf $\shf{I}$ is {\em of degree $(d,e)$} if $\shf{I}$ is isomorphic to $\oshf(-d,-e)$ as an $\oshf_{\pone \times \pone}$-module.
\end{deff}
\begin{rem}  We abbreviated the term ``effective'', the condition what we required for divisorial correspondence as a divisor on $\pone \times \pone$.
\end{rem}
\begin{rem} In a theory of quantum cohomology, quasimap is used to compactify moduli spaces of maps (\cite{Br}).  Divisorial quasi-correspondence is to divisorial correspondence what quasimap is to map. 
\end{rem}
\begin{prop} For the scheme $(X,\oshf_X )$ and an invertible sheaf $\lshf$ over $X$, let $\fct{P}(\lshf)$ be the set of the ideal sheaves on $X$ isomorphic to $\lshf$ as an $\oshf_X$-module. Then we have a bijection $\phi_{X,\lshf } : \fct{P}(\lshf) \simeq (\lshf^{\vee}(X)\setminus\{0\}) /\oshf(X)^\times$ which is natural in the following sense:

 For any morphism $f: Y \to X$ of schemes, the following diagram \[
 \begin{tikzcd}
P(\shf{L}) \arrow{r}{\phi_{X,\shf{L}}} \dar &  (\lshf^{\vee}(X)\setminus\{0\}) /\oshf(X)^\times \dar \\
 P(f^* \shf{L}) \arrow{r}{\phi_{Y, f^* \shf{L} }} & (f^*\lshf^{\vee}(Y)\setminus\{0\}) /\oshf(Y)^\times
 \end{tikzcd} \label{moduli}
 \]
is commutative.
\end{prop}
\begin{proof} We first construct a map $\pi : \lshf^{\vee}(X)\setminus\{0\} \to \fct{P}(\lshf)$. For a global section $s \in\lshf^{\vee}(X)$, we denote by $s_f : \oshf \to \lshf^{\vee}$ an injective sheaf morphism defined by $1|_U \mapsto s|_U$ at every open $U \subset X$. By $\shf{I}_f$ we denote the ideal sheaf defined by $i_s \otimes_{\oshf} \lshf : \lshf \to \oshf$. This gives a map $\pi : \lshf^{\vee}(X)\setminus\{0\} \to \fct{P}(\lshf)$.

We show $\pi$ is surjective. For a given ideal sheaf $\ishf \to \oshf$, we fix an isomorphism $\lshf \simeq \ishf$. By tensoring $\lshf^{\vee}$ to these morphisms, we obtain an injective morphism $\oshf \simeq \lshf^{\vee} \otimes \ishf \to \lshf^{\vee}$. By taking the global sections, we obtain a $\oshf(X)$-linear map $i : \oshf(X) \to \lshf^{\vee}(X).$
We set $s := i(1)$ and then $\shf{I} = \shf{I}_s$ follows immediately. Therefore $\pi$ is surjective.

For $c \in \oshf(X)^\times = \fct{Aut}_{\oshf_X}(\oshf)$, we have $\ishf_s = \ishf_{cs}$ as subsheaf of $\oshf$. Therefore the map $\phi_{X,\shf{L} }$ is well-defined and injective. On the other hand, if $\ishf_s = \ishf_s$ for $s,s' \in \lshf^{\vee}(X)$, then we have $i_s(\oshf) = i_s'(\oshf)$ as images on $\lshf^{\vee}$. Therefore the images on global sections are equal, that is, we have $\oshf(X) s = \oshf(X)s'$ as a sub $\oshf(X)$-module of $\lshf(X)$. Therefore, we have $\pi(s) = \pi(s') \Leftrightarrow \exists c \in \oshf(X)^\times,\  \ s = cs'$. Thus $\phi_{X, \shf{L} } $ is surjective.

To show the naturality, we have to show $\phi_{X, \shf{L}}(s) \otimes_{\oshf_X} \oshf_Y = \phi_{Y, \shf{L}}(s \otimes 1)$, for any $s \in \lshf^{\vee} (X) \setminus \{0 \}$. This is obvious from the definition of the morphism of the tensor product functor.
\end{proof}
\begin{cor}\label{ccorproj} Let $\ccor_{d,e}$ be the fine moduli space of divisorial correspondences of degree $(d,e)$. Then we have $\ccor_{d,e} \simeq (\af^{(d+1)(e+1)} \setminus \{0\})/\gmult \simeq \proj^{(d+1)(e+1)-1}$.
\end{cor}



\subsection{Resultant and Composition}
For a Laurent polynomial $f$ in a variable $x$, we write $f[x^i]$ for the coefficient of $x^i$.
\begin{deff} Let $R$ be a commutative ring $R$ and $x$ a variable. For $R[x]_d := \{ f \in R[x] \mid \deg_x f \leq d \}$, the resultant $\fct{res}_{x,(d,e)} : R[x]_d \times R[x]_e \to R$ is defined as the determinant of Sylvester matrix
\[
\fct{res}_{x,(d,e)}(f ,g) := \left| \begin{array}{ccccccccc}
f[x^0] & f[x^1] &      & \cdots &     f[x^d]     &      & &\\ 
     & f[x^0] & f[x^1] &           &  \cdots      & f[x^d] & &\\
     &      & \ddots &\ddots &         & \ddots   & \ddots &\\ 
     &      &      & f[x^0]      & f[x^1]        &           &  \cdots  & f[x^d] \\
g[x^0] & g[x^1] &    & \cdots &           g[x^e]     &      & &\\ 
     & g[x^0] & g[x^1]   &    &  \cdots  &         g[x^e] & &\\
     &      & \ddots &\ddots &         & \ddots     & \ddots &\\ 
     &      &      & g[x^0]      & g[x^1]        &          &  \cdots    & g[x^e]
\end{array} \right|  \]

For homogeneous polynomials of two variables $x,y$, $F(x,y) = x^df(\frac{y}{x})$ and $G(x,y) = x^ef(\frac{y}{x})$ of degree $d$ and $e$ respectively, we define the {\em homogeneous resultant} \[ \fct{res}_{[x,y]}(F(x,y) , G(x,y)) := \fct{res}_{x,(d,e)}(f(x),g(x)). \]
\end{deff}
\begin{rem}(\cite[Chapter 2.4]{Sil2}) For an integral domain $R$, the homogeneous resultant of two homogeneous polynomials on $R$ is $0$ if and only if the polynomials have a common factor as homogeneous polynomials over $\fct{Frac}(R)$.
\end{rem}

\begin{prop}\label{resequiv} For any regular matrix $g = \binom{a \ b}{c \ d} \in GL_2(\cpx )$ and two homogeneous polynomials $F(x,y) $ and $G(x,y)$, we have 
\[ \fct{res}_{[x,y]}(F(g\cdot (x,y)) , G(g \cdot (x,y))) = (\det g)^{\fct{deg} F \cdot \fct{deg} G }\fct{res}_{[x,y]}(F(x,y), G(x,y)), \]
where $g \cdot (x,y) := (ax+by , cx+dy)$.
\end{prop}
\begin{proof} See \cite[Chapter 12.1]{G-K-Z}.
\end{proof}

In the theory of arithmetic dynamics, a homogeneous resultant is familiar as a value which determines well-definedness of a rational self-map over $\pone$ and gives a Lipschitz constant with respect to the chordal metric (see \cite[Chapter 2.4]{Sil2}). In the rest of this section, we will see composition of correspondences is also expressed by the resultant. This expression roughly follows from the following observation: composition is substitution, substitution is elimination and elimination is expressed by (homogeneous) resultant.
\begin{prop} \label{affres} 
Let $C = \fct{Spec}k[x,y]/(f(x,y))$ and $D = \fct{Spec}k[x,y]/(g(x,y))$ be closed subschemes of $\af^2$ and $i_C$ and $i_D$ their inclusion morphisms respectively. Let $p : \widetilde{C \circ D} \to \af^2$ be the morphism defined as the composition of the morphisms in upper row of the following diagram
 \[
 \begin{tikzcd}
 \widetilde{C \circ D} \rar \dar \arrow[phantom]{dr}{\mathrm{p.b.}}& \af^3 
 \arrow{d}{\mathrm{id}\times \Delta \times \mathrm{id}} \arrow{r}{p_{13}}& \af^2 \\
 C \times D \arrow{r}{i_C \times i_D} & \af^4&
 \end{tikzcd}
 \] 
where $\Delta : \af^1 \to \af^2$ is the diagonal morphism. Then the scheme-theoretic image of $p$, denoted by $C \circ D \subset \af^2$, is written as 
\[ C \circ D \simeq \fct{Spec} k[x,y]/(\fct{res}_{z,(d,e)}(f(x,z),g(z,y)) ) \to \af^2. \]
\end{prop}
\begin{proof} The following diagrams
 \[
 \begin{tikzcd}
 C \arrow{r}{i_C} \arrow{d}[swap]{\fct{id} \times (p_2 \circ i_C)} & \af^2 \arrow{d}{\mathrm{id} \times \Delta} & & D \arrow{r}{i_D} \arrow{d}[swap]{(p_1 \circ i_D) \times \fct{id}} & \af^2 \arrow{d}{\Delta \times \fct{id}} \\
 C \times \af^1 \arrow{r}{i_C \times \fct{id} } & \af^3 & &  \af^1 \times D \arrow{r}{\fct{id} \times i_D} & \af^3
 \end{tikzcd}
 \]
are pullback diagrams. On the other hand, we have a pullback diagram
 \begin{equation}\label{icid}
 \begin{tikzcd} C \times D \arrow{d}{\fct{id} \times i_D} \arrow{r}{i_C \times \fct{id} } & C \times \af^2 \arrow{d}{\fct{id} \times i_C }\\ \af^2 \times D \arrow{r}{\fct{id} \times i_D} & \af^2 \times \af^2.
 \end{tikzcd}
 \end{equation}
By base-changing (\ref{icid}) by $\af^3 \xrightarrow{\fct{id} \times \Delta \times \fct{id} } \af^4$, we obtain the pullback diagram
\[ \begin{tikzcd} \widetilde{C \circ D} \arrow{r} \arrow{d} \arrow[phantom]{dr}{\mathrm{p.b.}} & C \times \af^1 \arrow{d}{i_C \times \fct{id} } \\
\af^1 \times D \arrow{r}{\fct{id} \times i_D} & \af^3.
 \end{tikzcd}
\]
This means $ \widetilde{C \circ D}$ is equal to $\fct{Spec} k[x,y,z]/ (f(x,z), g(z,y))$. The assertion follows from the property of resultant.
\end{proof}

\begin{prop}\label{corrcomp} Let $C,D \subset \pone \times \pone$ be divisorial correspondences, $F = (f_{ij}), \ G = (g_{kl})$ be the coefficient matrices of $C$ and $D$ respectively. Let us consider the following diagram
 \begin{equation}
 \begin{tikzcd}
  \widetilde{C \circ D} \rar \dar \arrow[phantom]{dr}{\mathrm{p.b.}}& \pone \times \pone \times \pone 
 \arrow{d}{\mathrm{id}\times \Delta \times \mathrm{id}} \arrow{r}{p_{13}}& \pone \times \pone \\
 C \times D \arrow{r}{\subset \times \subset} & \pone \times \pone \times \pone \times \pone& 
 \end{tikzcd}\label{comp}
 \end{equation}
 Then the composition of the morphisms in upper row $ \widetilde{C \circ D} \to \pone \times \pone$ factors through the divisorial correspondence given by the coefficient matrix $H = (h_{mn})$ such that
\[ \sum_{m,n} h_{mn}x_0^{dd'-m}x_1^{m}y_0^{ee'-n}y_1^{n} = \fct{res}_{[z_0,z_1]} \left( \sum_{i,j}f_{ij}x_0^{d-i}x_1^iz_0^{e-j}z_1^j , \sum_{k,l}g_{kl}z_0^{d'-k}z_1^{k}y_0^{e'-l}y_1^l \right) \]
if $H \neq 0$.
\end{prop}
\begin{proof}
For a standard open covering of $\pone$, $\{ U_0 = \pone \setminus \{ 0 \}, U_1 = \pone \setminus \{ \infty \}  \}$, we denote $U_{\alpha_1\alpha_2 \ldots \alpha_n}$ for the open subscheme $U_{\alpha_1} \times \cdots \times U_{\alpha_n}$ of $(\pone)^n$. The closed subscheme $\fct{Im}(\fct{id} \times \Delta \times \fct{id})$ of $(\pone)^4$ is covered by the open subschemes $U_{\alpha \beta \beta \gamma}$ of $(\pone)^4$. Proposition \ref{affres} gives the construction of the upper row of (\ref{comp}) over each $U_{\alpha \beta \beta \gamma} \simeq \aff^4$, where the codomain is restricted to $U_{\alpha \gamma} \simeq \aff^2 \subset \pone \times \pone$. By combining these constructions, we obtain the assertion.
\end{proof}
\begin{rem}\label{nonzerocor} In the situation of Proposition \ref{corrcomp}, the properties of resultants gives that we have $H = 0$ if and only if $f(x,y) = \bar{f}(x,y)(ay_0 + by_1)$ and $g(x,y) = (ax_0 + bx_1)\bar{g}(x,y)$ for some $a,b \in \bar{k}$ and some divisorial correspondences $\bar{f}$ and $\bar{g}$.
\end{rem}
\begin{deff} {\em The composition map} $ \circ  : \ccor_{d,e} \times \ccor_{d',e'} \dashrightarrow \ccor_{dd' , ee'} $ is the map that sends the pair of divisorial quasi-correspondence $(f(x,y) , g(x,y))$ to the quasi-correspondence $(f \circ g)(x,y) := \fct{res}_z(f(x,z) , g(z,y))$.
\end{deff}
\begin{prop} The composition map is associative, that is, the following diagram is commutative: 
\begin{equation}
\begin{tikzcd}
  \ccor_{d,e} \times \ccor_{d',e'} \times \ccor_{d'' , e''} \arrow[r,dashed]{}{\circ \times \mathrm{id}} \arrow[d,dashed]{}{\mathrm{id} \times \circ} & \ccor_{dd',ee'} \times \ccor_{d'', e''} \arrow[d,dashed]{}{\circ} \\ 
  \ccor_{d,e} \times \ccor_{d'd'',e'e''} \arrow[r,dashed]{}{\circ} & \ccor_{dd'd'',ee'e''}
\end{tikzcd}\label{assoccomp}
\end{equation}
\end{prop}
\begin{proof} By Remark \ref{nonzerocor}, the compositions of the two diagonal paths in (\ref{assoccomp}) are rational map. By Proposition \ref{corrcomp}, the images through the two paths of a general point of $\ccor_{d,e} \times \ccor_{d',e'} \times \ccor_{d'' , e''}$, which indicates the tuple of correspondences $(C,C',C'')$, are both given by the upper row of the following diagram:
\begin{equation*}
 \begin{tikzcd}
  \widetilde{C \circ C' \circ C''} \rar \dar \arrow[phantom]{dr}{\mathrm{p.b.}}& (\pone)^4 
 \arrow{d}{\mathrm{id}\times \Delta \times \Delta \times \mathrm{id}} \arrow{r}{p_{14}}& (\pone)^2 \\
 C \times C'  \times C'' \arrow{r}{\subset \times \subset \times \subset} & (\pone)^6&
 \end{tikzcd}
\end{equation*}
Therefore (\ref{assoccomp}) is commutative.
\end{proof}
\begin{deff} \em{The iteration map} $\Psi_n : \ccor_{d,e} \ratmap \ccor_{d^n , e^n}$ is the map that sends the divisorial quasi-correspondence $f(x,y)$ to the quasi-correspondence $(f \circ f \circ \cdots \circ f)(x,y)$. 
\end{deff}

\begin{prop}\label{conjform} Let $g \in \pglt \simeq \fct{Aut}(\pone)$ be a morphism given by $g ( [x_0 : x_1] ) = [d x_0 + cx_1 : b x_0 + ax_1]$ and $f(x,y)$ be a divisorial quasi-correspondence. Then we have
\begin{equation*} (g \circ f \circ g^{-1}) (x,y) =f (d x_0 +cx_1, b x_0 + a x_1 , d y_0 + cy_1 ,b y_0 + ay_1).
\end{equation*}
\end{prop}
\begin{proof}Let $g \in \pglt$ be the morphism given by $g ( [x_0 : x_1] ) = [d x_0 + cx_1 : b x_0 + ax_1]$. The graph correspondence $\Gamma_g$ is given by a divisorial quasi-correspondence
\[ G = G(x,y) = (bx_0 + ax_1)y_0 - (dx_0 + cx_1 )y_1 . \]
Let $f \in \ccor_{d,e}$ be a divisorial correspondence. We have
\begin{align*} \fct{res}_{[z_0 , z_1]} (f(x,z) , G(z,y)) & = \left| \begin{array}{cccc}
f[y^0] & f[y^1] & \cdots &     f[y^d]  \\
G[x^0] & G[x^1] & & \\
 & \ddots & \ddots & \\
 & & G[x^0] & G[x^1]
\end{array} \right| \\
&= \sum_{i = 0}^d (-1)^{i-1} f[y^i] G[x^1]^{d-i} G[x^0]^{i} \\
&= \sum_{i=0}^d f[y^i] G[x_1]^{d-i}(-G[x_0])^i \\
&= \sum_{i=0}^d f[y^i] (ay_0 - cy_1)^{d-i}(-b y_0 + dy_1 )^i \\
&= f(x_0,x_1,ay_0 - cy_1 , -by_0 + dy_1)
\end{align*}
and
\begin{align*} \fct{res}_{[z_0 , z_1]} (G(x,z) , f(z,y)) & = \left| \begin{array}{cccc}
G[y^0] & G[y^1] & & \\
 & \ddots & \ddots & \\
 & & G[y^0] & G[y^1] \\
f[x^0] & f[x^1] & \cdots &     f[x^d]  \\
\end{array} \right| \\
&= \sum_{i = 0}^d (-1)^{d+i-1} f[x^i] G[y^1]^{d-i} G[y^0]^{i} \\
&= \sum_{i=0}^d f[x^i] (-G[y_1])^{d-i}G[y_0]^i \\
&= \sum_{i=0}^d f[x^i] (d x_0 + cx_1)^{d-i}(b x_0 + ax_1 )^i \\
&= f(d x_0 + cx_1, b x_0 +a x_1 , y_0 , y_1).
\end{align*}
For the inverse element $g^{-1} \in \fct{PGL}_2$, the graph variety $\Gamma_{g^{-1}}$ is given by a divisorial quasi-correspondence
\[ G^{-1}(x,y) := (-bx_0 + dx_1)y_0 - (ax_0 - cx_1 )y_1. \]
Therefore we have
\begin{align*} (g \circ f \circ g^{-1}) (x,y) &= (g \circ f) (x_0,x_1, dy_0 + cy_1 ,by_0 + ay_1) \\
&= f (d x_0 +cx_1, b x_0 + a x_1 , d y_0 + cy_1 ,b y_0 + ay_1).
\end{align*}
\end{proof}
\begin{cor}\label{conjact}
For the conjugation action
\begin{equation} \ccor_{d,e} \times \pglt \to \ccor_{d,e} : (f,g) \mapsto g \circ f \circ g^{-1 },  \label{conju}
\end{equation}
the induced action of $\slt$ on $\ccor_{d,e} \simeq \proj (V_{d,e})$ is given by a representation on $V_{d,e}$ and the representation is isomorphic to $V_d \otimes V_e$.
\end{cor}
\begin{proof}
By Proposition \ref{conjform} the action of $\slt$ on $V_{d,e}$ is isomorphic to the one on the tensor space of the space of $d$-ary forms and the space of $e$-ary forms, $\fct{Sym}_d(V^*_1)^* \otimes \fct{Sym}_e(V^*_1)^*$. By Proposition \ref{dualrep}, it is isomorphic to $V_d \otimes V_e$.
\end{proof}

\section{Fundamental Properties of The Moduli of Correspondence}

In this section, we give simple generalizations of the results in \cite{Sil} and \cite{Le}, a characterization of the stable/semistable locus of the group action and the rationality of the moduli spaces.

\subsection{Stability of Group Action}
From now on until Theorem \ref{constquot}, we briefly review the geometric invariant theory \cite{MF}.
\begin{deff}(\cite[Definition 1.6]{MF}) Let $G$ be a reductive group scheme and $X$ a scheme with $G$-action $\sigma : G \times X \to X$. For an invertible sheaf $\lshf$ over $X$, an isomorphism $\phi : \sigma^* \lshf \simeq p_2^* \lshf$ is said to be {\em $G$-linearization} if $\phi$ satisfies the cocycle condition \[ p_{23}^* \phi \circ (\id_G \times \sigma)^*\phi = (\mu \times \fct{id}_X)^* \phi  \text{ (on }G \times G \times X).\]
\end{deff}
\begin{rem}\label{linzrep} If $\lshf$ is very ample and $G$ is affine, then $G$-linearization is described as the $G(\oshf(X))$-action on $\lshf(X)$ compatible with $\sigma$.
\end{rem}
\begin{rem}\label{Glinpow} For a $G$-linearization $\phi $ of an invertible sheaf $\lshf$ over a normal scheme $X$, $\phi^n : \sigma^* \lshf^n \simeq p_2^* \lshf^n $ is a $G$-linearization of $\lshf^n$.
\end{rem}
\begin{rem}(\cite[Proposition 1.4]{MF}) If there exists no surjective homomorphism $G \to \mathbb{G}_m$ of group schemes and $X \times_k \bar{k}$ is normal, $G$-linearization $\phi $ of an invertible sheaf $\lshf$ is unique if exists.
\end{rem}
%

For a given action and a given invertible sheaf, $G$-linearization may not be unique, for instance, if the action is trivial, any regular homomorphism $G \to \fct{Aut}(\lshf)$ gives a $G$-linearization.

\begin{deff}(\cite[Definition 1.7]{MF}) Let $G$ be a reductive group, $X$ an algebraic variety with $G$-action and $P$ a geometric point of $X$.
\begin{enumerate} \item $P$ is said to be {\em pre-stable} if the stabilizer group of $P$ is finite and there exists a $G$-stable affine open neighborhood of $P$.
\end{enumerate}
Moreover, we suppose that $\lshf$ be an ample invertible sheaf over $X$ with $G$-linearization.
\begin{enumerate}
\item $P$ is said to be {\em $\lshf$-semistable} if there exists $f \in H^0(X, \lshf^n)^G$ such that $f(P) \neq 0$ and $X_f$ is affine.
\item $P$ is said to be {\em (proper) $\lshf$-stable} if $P$ is $\lshf$-semistable and pre-stable.
\end{enumerate}
The set of pre-stable (resp. $\lshf$-semistable, $\lshf$-stable) geometric points is the set of geometric points of an open subscheme of $X$ called {\em pre-stable (resp. $\lshf$-semistable, $\lshf$-stable) locus}. We denote the loci by $X^s(\fct{Pre})$ (resp. $X^{ss}(\lshf )$, $X^s(\lshf)$).
\end{deff}
\begin{rem} For a $G$-variety $X$ which is isomorphic to a projective space $\proj (V)$, we sometimes write $X^s$ and $X^{ss}$ for the stable locus and semistable locus of any $\oshf(n)$ with $G$-linearization. 
\end{rem}
\begin{rem}(\cite[Converse 1.12]{MF})\label{uniqlin} If the categorical (resp. the geometric) quotient of $X$ by $G$ exists, then $X = X^{ss}(\shf{L})$ (resp. $X = X^s(\shf{L} )$) for some ample invertible sheaf $\lshf$ over $X$ with $G$-linearization.
\end{rem}
\begin{thm}\label{affquot}(\cite[Theorem 1.1]{MF}) Let $G$ be a reductive group, $X = \fct{Spec} R$ an affine algebraic variety with $G$-action. Then the categorical quotient $X \sslash  G$ is constructed as $\fct{Spec} R^G$.
\end{thm}
\begin{thm}(\cite[p.40]{MF})\label{constquot} Let $G$ be a reductive group, $X$ a proper algebraic variety with $G$-action, $\lshf$ a very ample invertible sheaf with $G$-linearization. Then the categorical quotient $X^{ss}(\shf{L} ) \sslash  G$ is constructed as $ \fct{Proj} \bigoplus_{i=0}^{\infty } H^0(\lshf^i , X)^G .$
\end{thm}
By Proposition \ref{ccorproj} and Corollary \ref{conjact}, the moduli space of correspondences $\ccor_{d,e}$ is isomorphic to $\proj(V_d \otimes V_e)$ equivariantly with respect to $\slt$-actions. By Remark \ref{linzrep}, an $\slt$-linearization is given by the natural representation over $H^0(\proj (V_d \otimes V_e) , \oshf (1) ) \simeq (V_d \otimes V_e)^*$, the dual representation of $V_d \otimes V_e$. The last representation is isomorphic to $V_d \otimes V_e$ by Proposition \ref{dualrep}. By Theorem \ref{constquot}, the uniform categorical quotient $\ccor_{d,e}^{ss} \sslash \slt$ is constructed as $\fct{Proj} k[V_d \otimes V_e]^{\slt }$.

The reductive group $\pglt$ also has a $G$-linearization on $\oshf (\frac{2}{\gcd(2,d+e)} )$ by Proposition \ref{CGdec} and Corollary \ref{conjact}. In fact, we have $ k[V]^{\slt } =  k[V]^{\pglt } $ for any finite dimensional representation $V$ of $\pglt$, therefore $\ccor_{d,e} \sslash \pglt \simeq \fct{Proj} \bigoplus_{i=0}^{\infty } k[V_d \otimes V_e ]_{2i }^{\pglt} \simeq \fct{Proj} \bigoplus_{i=0}^{\infty } k[V_d \otimes V_e ]^{\slt}$.
\begin{thm} A divisorial correspondence $C$ given by a quasi-divisorial correspondence $\sum_{i=0}^d \sum_{j = 0}^e a_{ij}x^iy^j $ is \emph{not} a stable point (resp. \emph{not} a semistable point) of $\ccor_{d,e}$ if and only if there exists an $\slt$-conjugate of the coefficient matrix $(b_{ij})$ such that $b_{ij} = 0$ for all $i+j < \frac{d+e}{2}$ (resp. $i+j \leq \frac{d+e}{2}$).
\end{thm}
\begin{proof}
By \cite{Sil}, any maximal subtorus of $\fct{SL}_2(k)$ is conjugate to $c : \gmult \to \fct{SL}_2(k)$ such that $c(t) := \binom{t\ \ 0\ }{0\ t^{-1}}$. For a divisorial quasi-correspondence $f$, we have
\[ c(t) \cdot f = f(tx_0,t^{-1}x_1 , ty_0 , t^{-1}y_1 ) = \sum_{i=0}^d \sum_{j = 0}^e t^{(d+e) - 2(i+j)}a_{ij}x^iy^j. \] Therefore by the numerical criterion (\cite[Theorem 2.1]{MF}), we obtain the claim.
\end{proof}

\subsection{Rationality}
Let $V = V' \oplus V''$ be a representation of a reductive group $G$. Then we have the inclusion morphism $k[V'] \to k[V]$. This morphism is $G$-equivariant by definition, therefore leads to the morphism $k[V']^G \to k[V]^G$ and the rational map $\proj (V^*)^{ss} \sslash  G \dashrightarrow \proj ((V')^*)^{ss} \sslash  G$. If the action of the group $G$ on $(V')^*$ is free for general point, then the fiber of general point of $\proj((V')^*)^{ss} \sslash  G$ via $\proj(V^*)^{ss} \sslash G \ratmap \proj ((V')^*) \sslash  G$ is naturally isomorphic to $(V/V')^*$. Therefore we have the following proposition.
\begin{prop}\label{ratfib} For a representation of a reductive group $G$ and a representation $V$, $\proj (V^*)\sslash G$ is rational if there exists a subrepresentation $V' \subset V$ such that the action of $G$ on $V'$ is generically free and $\proj((V')^*)\sslash G$ is rational.
\end{prop}

We fix a infinity field $k$. Let $1 \leq d \leq e$ be positive integers. Then the Clebsch-Gordan decomposition (Proposition \ref{CGdec}) gives
\[ \dim_k \fct{Hom}_{\fct{Rep}_k(\fct{SL}_2 )}( V_{d+e-1} \otimes V_1 , V_d \otimes V_e ) = 2. \]
We fix an injective morphism
\[ \rho_1 : V_{d+e-1} \otimes V_1 \to V_d \otimes V_e .\]
The morphism $\rho_1$ gives rational maps
\[ \rho_1 : \ccor_{d,e} \dashrightarrow \ccor_{1, d+e-1} \]
and
\[\bar{ \rho_1} : \dyn_{d,e} \dashrightarrow \dyn_{1,d+e-1}. \]
\begin{prop}\label{rational} $\dyn_{d,e}$ is rational for $d, e \geq 1$ and $(d,e) \neq (1,1)$.
\end{prop}
\begin{proof}
For the case $d=1$, this is shown in \cite{Le}. In the same paper, it is also shown $\slt(k)$ acts generically free on the representation $V_D \otimes V_1$ for $D \geq 3$. Therefore the general case follows from Proposition \ref{ratfib}.
\end{proof}
\begin{rem} In \cite{Ma1}, \cite{Ma2}, \cite{Ma3}, \cite{Ma4} and \cite{Sh}, the rationality of $\ccor_{d,e} \sslash  \fct{SL}_2 \times \fct{SL}_2$ is shown for some $(d,e)$'s. The rationality of $\fct{SL}_2$ leads to the rationality of $\ccor_{d,e}$ for these cases.
\end{rem}

\section{Multiplier Map}

\subsection{Construction}
Let $f$ be a divisorial quasi-correspondence \[ f(x_0,x_1,y_0,y_1) = \sum_{0 \leq i \leq d,\ 0 \leq j \leq e} a_{ij}x_0^{d-i}x_1^{i}y_0^{e-j}y_1^j. \]
We fix affine coordinates $x = \frac{x_1}{x_0}$ and $y = \frac{y_1}{y_0}$. Then the restriction of the divisorial quasi-correspondence given by $f$ on the affine chart $\aff_x \times \aff_y = U_{\pone_x}^+ (x_0) \times U_{\pone_y}^+ (y_0)$ of $\pone_x \times \pone_y$ is given by $V( \bar{f}(x,y) ) $, where \[ \bar{f}(x,y) := \sum_{0 \leq i \leq d ,\ 0 \leq j \leq e} a_{ij}x^i y^j. \] Intuitively, from implicit function theorem, the derivative (we denote by) $\frac{dy}{dx}$ on $V(\bar{f}(x,y))$ around a point $(\bar{x} ,\bar{y} ) \in V(\bar{f}(x,y))$ is given by 
\begin{equation} \frac{dy}{dx} = - \left. \frac{\partial_x \bar{f}(x,y)}{ \partial_y \bar{f}(x,y)} \right|_{(x,y) = (\bar{x} , \bar{y} )}.  \label{impthm}
\end{equation}
if $ \left. \partial_y \bar{f}(x,y) \right|_{(x,y) = (\bar{x} , \bar{y} )} \neq 0$. In fact, from the universality of differential sheaf, if we change affine coordinates of each components equivariantly, the value $\frac{dy}{dx}$ does not changes.
We use (\ref{impthm}) as the definition of the derivative $\frac{dy}{dx}$ of $f$ at $(\bar{x},\bar{y}) \in V(f) \setminus V(\partial_yf)$.
We note that here
\begin{align*} x_0^dy_0^e \partial_x \bar{f}(x,y) &= \partial_{x_1}f(x, y) = d\cdot  f (x,y) - \partial_{x_0}f(x,y)\ \text{and} \\
 x_0^dy_0^e \partial_y \bar{f}(x,y) &= \partial_{y_1}f(x, y) = d\cdot  f (x,y) - \partial_{y_0}f(x, y) \end{align*}
as rational functions of $x_0, x_1, y_0$ and $y_1$. Therefore, at a point $(\bar{x},\bar{y}) \in V(f) \setminus V(\partial_yf)$, we have
\begin{align*}
\frac{dy}{dx} = -\frac{ \partial_{x_1}f(\bar{x},\bar{y})}{ \partial_{y_1}f(\bar{x},\bar{y})} = -\frac{ -\partial_{x_0}f(\bar{x},\bar{y})}{ -\partial_{y_0}f(\bar{x},\bar{y})}. \end{align*}
For convenience, we define rational functions $\partial_x f(x,y) $ and $\partial_yf(x,y)$ as 
\begin{align} x \partial_x f(x,y) &= x_1 \partial_{x_1}f(x,y) - x_0\partial_{x_0}f(x,y) \text{ and } \nonumber \\  
 y \partial_y f(x,y) &= y_1 \partial_{y_1}f(x,y) - y_0\partial_{y_0}f(x,y). \label{chain} \end{align}

We construct the symmetric form of the fixed point multipliers $\sigma_i(f)$ by
\begin{equation}
1 + \sum_{i=1}^{d+e} (-1)^i \sigma_i(f) t^i = \prod_{z : f(z,z) = 0}\left( 1 + \frac{\partial_x f(z,z)}{\partial_y f(z,z)} t \right). \label{symformmult}
\end{equation} 
The map $\ccor_{d,e} \dashrightarrow \af^{d+e}$ given by $f \mapsto (\sigma_i(f) )_{i=1,\ldots d+e}$ is a morphism if $d=1$ and $f$ is a morphism, but is just a rational map for $d \geq 2$ because there are points such that $\partial_y f (z) = 0$. To incorporate $y$-critical points, then we prefer to consider the following homogenized form of the multiplier map:
\begin{deff} \label{multmap}
{\em The fixed point multiplier map} is the rational map
\[ \lambda_{1,(d,e)} : \ccor_{d,e} \ni f \mapsto [\fct{res}_z (f(z,z) , \partial_x f(z,z) dx + \partial_y f(z,z)dy )] \in \proj (D_{d+e}), \]
where we regard $dx$ and $dy$ as just a variables and $D_{n}$ be the space of $n$-ary forms in them.

{\em The $n$-th multiplier map} is the rational map $\lambda_{n ,(d,e)} := \lambda_{1,(d^n , e^n)} \circ \Psi_n$. 
\end{deff}
\begin{rem}
By some elementary transformation of Sylvester matrix, we have a fundamental property of resultant
\[ \fct{res}_x (f(x) , g(x) + af(x) ) = \fct{res}_x (f(x) , g(x)) \]
for any $f , g \in k[x]$ and $a \in k$.
For any divisorial quasi-correspondence $f(x,y)$, this property of resultant and $x_1\partial_{x_1} f(x,y) = df(x,y) - x_0\partial_{x_0}f(x,y)$ shows the equality
\[ [\fct{res}_z (f(z,z) , \partial_x f(z,z) dx + \partial_y f(z,z)dy )] = [\fct{res}_z (f(z,z) , z_0 \partial_{x_0} f(z,z) dx + z_0\partial_{y_0} f(z,z)dy )] \]
of points of $\proj (D_{d+e})$.
\end{rem}

\subsection{Resultant Form of Woods Hole Formula}
Let $f$ be a divisorial quasi-correspondence of degree $(d,e)$. Then we have
\begin{align*}
z(\partial_x f) (z,z) & = \sum_{i,j} (d-2i) a_{ij}z^{i+j} \text{ and } \\
z(\partial_y f)(z,z) & = \sum_{i,j} (e-2j) a_{ij}z^{i+j}
\end{align*}
from (\ref{chain}) and 
\begin{align*}
z\partial_z (\Omega^0f)(z) & = \sum_{i,j}(d+e -2i-2j)a_{ij}z^{i+j} \text{ and }\\
z_0z_1(\Omega^1 f)(z) & = \sum_{i,j} (-ei + dj)a_{ij}z^{i+j}
\end{align*}
from the definition of Cayley operator (\ref{omega}).
This leads to
\[\left( \begin{matrix} z \partial_z (\Omega^0 f)(z) \\ z_0z_1 (\Omega^1f)(z)  \end{matrix}\right) = \left( \begin{matrix} 1 & 1 \\ \frac{e}{2} & \frac{-d}{2} \end{matrix} \right) \left(\begin{matrix} z(\partial_x f)(z,z) \\ z(\partial_y f)(z,z) \end{matrix}\right)  \]
and
\begin{eqnarray*} \left( \begin{matrix}z (\partial_x f)(z,z) \\ z(\partial_y f)(z,z)  \end{matrix} \right) &= \frac{-2}{d+e} \left(\begin{matrix} \frac{-d}{2} & -1 \\ \frac{-e}{2} & 1 \end{matrix}\right) \left(\begin{matrix}z \partial_z (\Omega^0 f)(z) \\ z_0z_1 (\Omega^1f)(z) \end{matrix}\right) \\
&=  \left(\begin{matrix} \frac{d}{d+e} & \frac{2}{d+e} \\ \frac{e}{d+e} & -\frac{2}{d+e} \end{matrix}\right) \left(\begin{matrix}z \partial_z (\Omega^0 f)(z) \\ z_0z_1 (\Omega^1f)(z) \end{matrix}\right). \\
 \end{eqnarray*}
This coordinate change gives the following proposition.
\begin{prop}\label{unif1} There exists a projective linear morphism $A \in \fct{Aut} (\proj^{d+e}) = \fct{PGL}_{d+e}(k)$ which makes the following diagram commutative:
\[
\begin{tikzcd} \fct{Dyn}_{d,e} \arrow[rr,dashed]{}{ \lambda_{1,(d,e)} } \arrow[d,dashed]{}{\bar{\rho}_1} & & \proj(D_{d+e}) \arrow{d}{A} \\
\fct{Dyn}_{1,d+e-1} \arrow[rr,dashed]{}{\lambda_{1,(1,d+e-1)}} & & \proj(D_{d+e}).
\end{tikzcd}
\]
\end{prop}
\begin{proof}
 By the definition of $\rho_1$, we have $\Omega^i \rho_1 (f) = c_i \cdot \Omega^i f$ for any $f \in \ccor_{d,e}$ and each $i = 0,1$, where $c _i \in k^\times$ are given. Therefore for a correspondence $f \in \ccor_{d,e}$ we have 
\begin{eqnarray*} \lambda_{1, (1, d+e-1)} \circ \rho_1 (f) = \fct{res}_z ( \Omega^0f(z) ,& c_0 z\partial_z (\Omega^0 f)(z)\left( \frac{1}{d+e}dx+\frac{d+e-1}{d+e}dy \right)  \\
 &  + c_1 z_0z_1 \Omega^1f(z) \left( \frac{2}{d+e}dx - \frac{2}{d+e}dy \right) )
\end{eqnarray*}
and
\begin{eqnarray*} \lambda_{1, (d, e)} (f) = \fct{res}_z ( \Omega^0f(z) , &  z\partial_z (\Omega^0 f)(z)\left( \frac{d}{d+e}dx+\frac{e}{d+e}dy \right)  \\
 &  +  z_0z_1 \Omega^1f(z) \left( \frac{2}{d+e}dx - \frac{2}{d+e}dy \right)  ) \end{eqnarray*}
for any correspondence $f$. We take the basis $(dz^0_{(d',e')} , dz_{d'+e'}^1 )$ of $D_1 = k \cdot dx \oplus k \cdot dy$ for each pair of positive integers $d',e'$ by 
\begin{equation} \label{dzdef} dz^0_{(d',e')} := \frac{d'}{d'+e'}dx+\frac{e'}{d'+e'}dy \text{ and } dz_{d'+e'}^1 := \frac{2}{d'+e'}dx - \frac{2}{d'+e'}dy. \end{equation}
The automorphism $A$ on $\proj (D_{d+e})$ is induced by the linear map 
\[ \left( \begin{matrix}dz^0_{(d,e)} \\  dz_{d+e}^1 \end{matrix} \right) \mapsto \left( \begin{matrix} c_0 dz^0_{(1,d+e-1)} \\ c_1 dz_{d+e}^1  \end{matrix} \right). \]
\end{proof}
\begin{rem} \label{multdeg} For a correspondence $f = \sum_{i,j} a_{ij}x^iy^j $, we have
\[ \lambda_{1, (d, e)} (f) = [\fct{res}_z ( \Omega^0f(z) ,  z\partial_z (\Omega^0 f)(z)dz^0_{(d,e)}  +  z_0z_1 \Omega^1f(z) dz^1_{d+e} )] \in \proj(D_{d+e}). \]
From the definition of resultant, we have each coefficient is divisible by $a_{00}a_{de}$.
As a point of $ \proj(D_{d+e})$, we have
\begin{equation} \lambda_{1, (d, e)} (f) = [\fct{res}_z ( \Omega^0f(z) , z\partial_z (\Omega^0 f)(z)dz^0_{(d,e)}  +  z_0z_1 \Omega^1f(z) dz^1_{d+e} )/a_{00}a_{de}]. \label{invmult} \end{equation}
We write $[(dz^0)^{d+e-i} (dz^1)^i]$ for the homogeneous coordinate of $\proj(D_{d+e})$ given by taking the coefficient of $(dz^0_{d,e})^{d+e-i} (dz^1_{d+e})^i$ for any nonzero element of $D_{d+e}$ as a polynomial of variables $dz^0_{d,e}$ and $dz^1_{d+e}$ defined in (\ref{dzdef}).
In this form, we have the coordinate $[(dz^1_{d+e})^{d+e}]$ of (\ref{invmult}) is
\[ \fct{res}_z(\Omega^0f(z) , z_0z_1 \Omega^1f(z) )/a_{00}a_{de} = \fct{res}_z(\Omega^0f(z) , \Omega^1f(z) ),\]
therefore $\slt$-invariant by Proposition \ref{resequiv}. From the invariance of the multiplier map, the other coordinate functions of (\ref{invmult}) are $\slt$-invariant on $\ccor_{d,e}$ of degree $2(d+e-1)$.
\end{rem}
\begin{thm}\label{hypcoord} The image of $\lambda_{1,(d,e)} : \dyn_{d,e} \to \proj (D_{d+e})$ is the hyperplane defined by
\[ [(dz^0)^{d+e-1} (dz^1)^1] = 0, \]
where $[(dz^0)^{d+e-1} dz^1]$ is the coordinate function of $(dz_{(d,e)}^0)^{d+e-1} dz_{d+e}^1$.
\end{thm}
\begin{proof}
By Proposition \ref{unif1}, it is enough to show the assertion for $e=1$, the case of the moduli of maps. We show the index theorem (\ref{index}) is equivalent to the hyperplane $ [(dz^0)^{d} (dz^1)^1] = 0 $ by the coordinate change.
The coordinate change is induced by the coordinate change
\[ \left(\begin{matrix} dz^0 \\ dz^1 \end{matrix} \right) = \left(\begin{matrix} \frac{d}{d+1} & \frac{2}{d+1} \\ \frac{1}{d+1} & -\frac{2}{d+1} \end{matrix}\right) \left(\begin{matrix} dx \\ dy \end{matrix} \right) \]
on $D_1$, which induces a coordinate change
\[ \left(\begin{matrix} [dz^0] \\ [dz^1] \end{matrix} \right) = \left(\begin{matrix} \frac{d}{d+1} & \frac{2}{d+1} \\ \frac{1}{d+1} & -\frac{2}{d+1} \end{matrix}\right)^{-1} \left(\begin{matrix} [dx] \\ [dy] \end{matrix} \right) = \left( \begin{matrix} 1 & 1 \\ \frac{1}{2} & \frac{-d}{2} \end{matrix} \right)\left(\begin{matrix} [dx] \\ [dy] \end{matrix} \right) \]
on $D_1^*$.

We note that $D_n$ is naturally isomorphic to $(\fct{Sym}_n(D_1^*))^*$ and therefore $D_n^*$ is naturally isomorphic to $\fct{Sym}_n(D_1^*)$. The vector space of $n$-ary form of $D_1^*$ is isomorphic to the vector space $(\fct{Sym}_{n}(D_1^*)^*)^* \simeq \fct{Sym}_{n}(D_1)^*$ and the canonical morphism $\beta_n : \fct{Sym}_{n}(D_1^*) \to (\fct{Sym}_{n}D_1)^*$ is given by $\beta_n([(ds)^i (dt)^{n-i}]) = \binom{n}{i} [ds]^i [dt]^{n-i}$ for any basis $\{ ds , dt \}$ of $D_1$ (Remark \ref{caniso}). Therefore we have
\begin{align*}
\beta_{d+1}( [(dz^0)^d (dz^1) ] ) &= \binom{d+1}{1} [dz^0]^d [dz^1] \\
&= (d+1) ( [dx] + [dy] )^d (\frac{1}{2} [dx] - \frac{d}{2}[dy]) \\
&= \frac{d+1}{2} \left( \sum_{i = 0}^{d+1} \left( d\cdot \binom{d}{i-1} - \binom{d}{i} \right) [dx]^i [dy]^{d+1-i} \right) \\
&= \frac{d+1}{2} \left( \sum_{i = 0}^{d+1} (i-1)\binom{d+1}{i} [dx]^i [dy]^{d+1-i} \right) \\
&= \beta_{d+1} \left( \frac{d+1}{2} \sum_{i=0}^{d+1} (i-1) [(dx)^i (dy)^{d+1-i}] \right),
\end{align*}
where we used
\begin{align*} d\cdot \binom{d}{i-1} - \binom{d}{i} &= \binom{d+1}{i}\left(d\frac{i}{d+1} - \frac{d+1-i}{d+1} \right) = (i-1)\binom{d+1}{i}
\end{align*}
and $\binom{d}{i} = 0$ for $i<0 , d<i$.

Therefore the hyperplane defined by $[(dz^0)^d (dz^1) ] = 0$ is the one defined by $\sum_{i=0}^{d+1} (i-1) [(dx)^i (dy)^{d+1-i}]=0 $. By (\ref{symformmult}) and Definition \ref{multmap}, this is which the index theorem (\ref{index}) defines.

The surjectivity of the multiplier map to the hyperplane is shown in \cite{Gor2}.
\end{proof}
This reduction is a representationally simplified form of Woods Hole Formula. 
\begin{cor}\label{repWH}For an arbitrary field $K$ and any polynomials $f,g \in K[x]$ of degree $\deg f \geq 3$ and $\deg f \geq \deg g +2$, 
\[ \left. \frac{\partial }{\partial t} \fct{res}_x(f(x) , f'(x) + t g(x)) \right|_{t = 0} = 0. \]
\end{cor}
\begin{proof} We put $d := \deg f$. By Proposition \ref{CGdec}, there exists a divisorial quasi-correspondence $\widetilde{f}(x,y)$ of degree $(d-1,1)$ such that $\Omega^0 \widetilde{f} (z) = f(z)$ and $\Omega^1 \widetilde{f}(z) = g(z)$. By Remark \ref{multdeg} and Theorem \ref{hypcoord}, we have the $k$-coefficient of $t$ of the polynomial $ \fct{res}_z (f(z) , zf'(z) + tzg(z) ) \in k[t]$ is $0$. Since resultant is $\zahl$-polynomial of the coefficients of $f$ and $g$, the $K[z]$-coefficient of $t$ is $0$ in arbitrary characteristics. This leads the assertion.
\end{proof}

\end{document}